\documentclass[12pt,a4paper]{amsart}
\usepackage{amsmath, amsfonts, amssymb,amsthm}
\usepackage{graphicx}
\usepackage{mathrsfs}
\usepackage{hyperref}
\usepackage{color}
\usepackage[dvipsnames]{xcolor}
\usepackage{etoolbox}
\usepackage{enumitem}

\newtheorem{thm}{Theorem}[section]
\newtheorem{defi}[thm]{Definition}
\newtheorem{pro}[thm]{Proposition}

\newtheorem{lem}[thm]{Lemma}

\numberwithin{equation}{section}
\usepackage{hyperref}

\newcommand{\R}{\mathbb{R}}

\newcommand{\D}{\displaystyle}
\def\qq#1{\qquad \mbox{#1}\quad}

\newcommand{\al}{\alpha}

\newcommand{\de}{\delta}

\newcommand{\Ga}{\Gamma}
\newcommand{\la}{\lambda}

\newcommand{\Om}{\Omega}
\newcommand{\Omb}{\overline{\Om}}
\newcommand{\p}{\partial}

\title[Uniform a  priori estimates]{Uniform a  priori estimates for slightly subcritical fractional problems}
\author{Juan-Carlos Felipe-Navarro}
\address{J.C. Felipe-Navarro
\newline
Departamento de An\'alisis Matem\'atico y Matem\'atica Aplicada,  Universidad Complutense de Madrid, 28040--Madrid, Spain.}
\email{jcfelipe@ucm.es}

\author{Rosa Pardo}
\address{R.~Pardo
\newline
Departamento de An\'alisis Matem\'atico y Matem\'atica Aplicada,  Universidad Complutense de Madrid, 28040--Madrid, Spain.}
\email{rpardo@ucm.es}
\thanks{The first author is supported by the Spanish grants PID2021-123903NB-I00 and RED2022-134784-T funded through MCIN/AEI/10.13039/501100011033 and ERDF “A way of making Europe”.
The second author is supported by grant PID2022-137074NB-I00,  MICINN,  Spain. Furthermore, both authors are supported by the Spanish grant PID2025-168009NB-I00 and are members of CADEDIF-UCM Research Group (920894).}
\date{}

\begin{document}

\begin{abstract}
We study the uniform $L^\infty(\Om)$ \emph{a priori} boundedness of positive weak solutions to the fractional semilinear Dirichlet problem $(-\Delta)^s u = f(u)$ in a bounded, convex, $C^{1,1}$ domain $\Omega \subset \mathbb{R}^N$ with homogeneous exterior conditions $u\equiv 0$ in $\R^N\setminus\Om$. We consider {\it slightly superlinear} nonlinearities of the form $f(t) = t^q L(t)$, where $1 \le q \le \frac{N+2s}{N-2s}$ and $L$ is a slowly varying function. Although uniform estimates are well-established in the strictly subcritical regime ($q < \frac{N+2s}{N-2s}$), the slightly subcritical case ($q = \frac{N+2s}{N-2s}$) is highly challenging due to the potential formation of bubbling profiles. 

In this work, we isolate a structural condition on the slowly varying perturbation, namely
$$
\lim_{t \to \infty} \frac{t \, |L'(t)|}{L^{\frac{N}{2s}}(t)} = \infty,
$$
which acts as an asymptotic barrier that prevents mass concentration. Under this assumption, we establish global uniform $L^\infty(\Om)$ bounds for positive solutions, significantly expanding the class of known nonlinearities for which such estimates hold.
\end{abstract}

\maketitle



\maketitle

\section{Introduction}

The study of a priori bounds for positive solutions to elliptic equations is a longstanding open problem in nonlinear analysis, primarily due to its role in establishing existence results via topological degree theory and fixed-point theorems. In this work, we treat the uniform $L^\infty(\Om)$ boundedness of positive weak solutions to the fractional Dirichlet problem
\begin{equation}\label{eq:ell:pb}
	\begin{cases}
		(-\Delta)^s u = f(u), &\qquad \text{in }\, \Om,  \\
		\qquad\ \, \ u = 0, &\qquad  \text{in }\, \R^N\setminus\Om,
	\end{cases}
\end{equation}
where $s \in (0,1)$, $N > 2s$, and $\Omega \subset \mathbb{R}^N$ is a bounded convex domain with a $C^{1,1}$ boundary. 

Here, the fractional Laplacian, $(-\Delta)^s$, is an integro-differential operator defined as
$$ (-\Delta)^s \, u(x) = C_{N,s}\ \mbox{P.V.}\int_{\R^N} \frac{u(x) - u(y)}{|x - y|^{N+2s}}\ dy, $$
where P.V. stands for the Cauchy principal value, and $C_{N,s}$ is a positive normalization constant depending only on the dimension $N$ and the fractional power $s$ that plays no role in our analysis.

In particular, we focus on nonlinearities $f$ that exhibit slightly subcritical growth. That is,
$$ 
\lim_{t \to \infty} \frac{f(t)}{t^{2^*_s-1}} = 0, 
$$
in terms of the fractional Sobolev exponent $2^*_s := \frac{2N}{N-2s}$.

In this context, a function $u \in H_0^s(\Omega)$ is said to be a weak solution to \eqref{eq:ell:pb} if
$$ \frac{C_{N,s}}{2} \iint_{\mathbb{R}^{2N}} \frac{(u(x)-u(y))(\phi(x)-\phi(y))}{|x-y|^{N+2s}} \, dx \, dy = \int_{\Omega} f(u(x)) \, \phi(x) \, dx,$$
for every test function $\phi \in H_0^s(\Omega)$, where
\begin{align*}
	H_0^s(\Omega) = \Big\{ &w\in L^2(\Omega) \ : \ w\equiv 0 \ \text{ in } \R^N\setminus \Omega \  \text{ and } \ [u]_{H^s(\mathbb{R}^N)} < +\infty  \Big\}.
\end{align*}
Here, the fractional Sobolev seminorm is given by
$$ [u]_{H^s(\Omega)}^2:=\iint_{\Omega \times \Omega} \frac{|w(x)-w(y)|^2}{|x-y|^{N+2s}} dx\,dy. $$

A fundamental point in this theory is the distinction between the boundedness of a single solution and the uniform boundedness of the entire set of solutions. While the individual estimates provide the regularity of solutions, the importance of uniform $L^\infty(\Om)$ estimates lies in their application to fixed-point theory and topological degree methods. For instance, uniform bounds are the key ingredient in defining the Leray-Schauder degree \cite{Leray-Schauder}, and ensuring the existence of positive solutions in both classical and fractional frameworks; see \cite[\S 2]{deF_Lion_Nus_1982, Cossio_Herron_Velez} and
\cite{Chen_Li_Li_IJM_2016, Barrios_DelPezzo_GMelian_Quaas}, respectively.

In the local problem for the Laplacian, thanks to the Caldéron-Zygmund and Schauder theories, it is enough to show that weak solutions are in $L^q$ for every $q\in [1,+\infty)$ to guarantee the regularity. For $N=1$ and $N=2$, the Sobolev embeddings ($H^1 \hookrightarrow C^{0,1/2}$ and $H^1 \hookrightarrow L^q$ for all $q < \infty$, respectively) are strong enough to provide the desired inclusion. However, for $N > 2$, the critical Sobolev exponent~$2^*-1 = \frac{N+2}{N-2}$ marks a structural boundary. That is, under the growth condition
$$ \limsup_{t \to \infty} \frac{f(t)}{t^{2^*-1}} < \infty,  $$
the $L^q$-integrability of weak solutions follows from the classical results of Moser \cite{Moser} and Br\'ezis and Kato \cite{BrezisKato1979} (see also \cite{Marino-Winkert_2019} and \cite[Appendix~B]{Struwe2008} for modern treatments).

The critical exponent  is not only an obstruction for the regularity of solutions but also a fundamental threshold for their existence. This phenomenon is deeply linked to the geometry of the domain through the Pohozaev identity \cite{Pohozaev}. Indeed, if $f(u) = |u|^{p-1}u$ with $p > 2^*-1$, the identity ensures that the only solution to the semilinear problem in a star-shaped domain is the trivial one, $u \equiv 0$.

Regarding the uniform bounds, Gidas and Spruck \cite{Gidas_Spruck_bd} established uniform a priori bounds in the subcritical regime
$$\lim_{t \to \infty} \frac{f(x,t)}{t^q}=h(x) \quad \text{uniformly in }x\in\Omb,\ \text{for some } q \in [0, 2^*-1),$$
where $h$ is continuous and strictly positive in $\Omb,$
by using the blow-up method. This technique relies on assuming the existence of an unbounded sequence of positive solutions $u_k$ to reach a contradiction via non-existence results. More precisely, one defines the rescaled sequence
$$
v_k(y) := M_k^{-1} u_k(x_k + \delta_k y), \quad \text{ in } \quad \Omega_k := \frac{1}{\delta_k}(\Omega - x_k),
$$
with $M_k=u_k(x_k)=\|u_k\|_\infty$, and  either $\de_k=\big(\frac{M_k}{f(M_k)}\big)^{1/2}$
or $\de_k=M_k^{-\frac{p-1}{2}}$. In the subcritical regime, passing to the limit maps the problem onto the whole space $\mathbb{R}^N$ or the half-space $\mathbb{R}^N_+$, where non-existence results (frequently known as Liouville-type theorems), ensures that the only solution is $v \equiv 0$ , which contradicts the normalization $v_k(0) = 1$. Nevertheless, this method fails at the critical exponent (as well as in the slightly subcritical regime) due to the appearance of non-trivial global solutions of the equations $-\Delta v = v^{2^*-1}$ or $-\Delta v = 0$ in $\mathbb{R}^N$. In the first case, Gidas, Ni, and Nirenberg \cite{Gidas_Ni_Nirenberg} proved that the only classical positive solutions with finite energy are the so-called standard bubbles (or Talenti functions)
\begin{equation*}
U_{\la}(x) = \la^{N-2}\,(\lambda^2 + |x|^2)^{-\frac{N-2}{2}}, \quad \la = \sqrt{N(N-2)}.
\end{equation*}
In the second case, the Liouville Theorem for harmonic functions \cite{Axler_Bourdon_Ramey, Nelson} dictates that $v \equiv 1$. In either scenario, such explicit solutions allow the $L^\infty$ norm to escape to infinity while maintaining a constant energy profile, thereby breaking the classical non-existence contradiction argument.

In \cite{deF_Lion_Nus_1982}, de Figueiredo, Lions, and Nussbaum obtained uniform a priori bounds by employing a completely different strategy under the structural hypothesis
\begin{equation} \label{CondDFLN}
    \liminf_{t \to +\infty} \frac{\theta \,F(t)-t\,f(t)}{t^2\,f(t)^{2/N}} \geq 0, \qquad \text{for some } \, \theta \in [0,2^*),
\end{equation}
where
$$  F(t) = \int_0^t f(\tau) \, d\tau.$$
They first establish a uniform $L^\infty(\Om)$ bound in a neighborhood of the boundary by combining a uniform $L^1$ estimate in the spirit of Br\'ezis and Turner \cite{Brezis_Turner} with the moving planes method of Gidas, Ni, and Nirenberg \cite{Gidas_Ni_Nirenberg}. Next, they exploit the Pohozaev identity to derive a global uniform $H^1_0(\Om)$ bound. Then, a bootstrap device due to Br\'ezis and Kato \cite{Brezis_Kato} is applied to elevate this energy estimate into a uniform $L^q$ bound for every $q < +\infty$. Finally, standard elliptic regularity concludes the uniform $L^\infty(\Om)$ estimate.

Although condition~\eqref{CondDFLN} does not cover the ``slightly subcritical'' regime, the authors conjectured that uniform bounds should still hold for such growth rates. This conjecture has been extensively addressed by the second author and collaborators (see \cite{Castro_Pardo_RMC_2015, Cuesta_Pardo_MedJM}). By estimating the radius of a ball where the solution exceed half of its $L^\infty(\Om)$ norm, we show that for certain perturbations of the critical power by slowly varying functions
$$ f(t)=t^{\tfrac{N+2}{N-2}}\,L(t),$$
the mass concentration required for bubbling is impossible, establishing then the uniform estimates.

The mathematical framework to describe these asymptotic perturbations is provided by the theory of regular variation in the sense of Karamata~\cite{Karamata}. 
\begin{defi} \label{Def:RVandSV}
A measurable positive function $g: (0, \infty) \to (0, \infty)$ is said to be a regularly varying function (at infinity) of index $q \in \mathbb{R}$ if
\begin{equation}\label{reva}
\lim_{t \to +\infty} \frac{g(\tau t)}{g(t)} = \tau^q, \quad \text{ for all } \ \tau > 0.    
\end{equation}
In this case, we write $g \in RV_q$. In particular, regularly varying functions of index $0$ are called slowly varying (at infinity), and the class is denoted by $SV$.
\end{defi}

These classes, introduced in the seminal work of Karamata \cite{Karamata}, allows for a systematic treatment of nonlinearities that deviate from pure power laws by logarithmic or iterated logarithmic factors, such as $L(t) = (\ln t)^\beta$ or $L(t) = (\ln(\ln t))^\beta$. For an exhaustive account of the properties of these functions, we refer the reader to the classical monographs \cite{Bingham, Resnick, Seneta}. In the context of nonlinear partial differential equations, this framework has proven to be extremely fruitful for analyzing problems with critical or nearly-critical growth (see, for instance, \cite{Costa_Quoirin_Tehrani, G-Melian_Iturriaga_RQuoirin}).
\medskip

The transition from local to non-local diffusion introduces significant analytical hurdles. As in the classical case, the regularity of each solution to problem~\eqref{eq:ell:pb} is well-understood. Weak solutions $u \in H_0^s(\Omega)$ to problems with critical growth are actually bounded and smooth. The step from $H_0^s(\Omega)$ to $L^\infty(\Omega)$ for the critical case can be found in \cite[Proposition 2.2]{Barrios_Colorado_Servadei_Soria}, while the subsequent Hölder regularity up to the boundary and interior classical regularity follow from the results by Ros-Oton and Serra \cite{RosOton_Serra_reg}.

Concerning the existence of uniform estimates in the non-local setting, the critical exponent $2^*_s$ marks the threshold for mass concentration. For strictly subcritical growth, the classical blow-up method of Gidas and Spruck has been successfully adapted to fractional equations \cite{Chen_Li_Li_AM_2017, Barrios_DelPezzo_GMelian_Quaas, Barrios_DelPezzo_GMelian_Quaas_DCDS_2017} as well as to systems of fractional equations \cite{Quaas_Xia_Nonl_2016, Lin_DCDS_2019, Wang_Niu}. As previously described in the classical scenario, non-existence results (usually known as Liouville-type theorems)  and the existence of bubbles   the crucial ingredients and the main obstructions, respectively, for this approach; see \cite{Chen_Li_Ou2006, Quaas_Xia_CVPDE, Jin_Li_Xiong_2014}.

In this paper, we extend to the fractional Laplacian the philosophy of \cite{Castro_Pardo_RMC_2015, Cuesta_Pardo_MedJM} to treat the ``almost critical'' case, which has remained largely unexplored. This method allows us to show that the interaction between the non-local operator and the slowly varying perturbation $L(t)$ prevents the formation of the aforementioned bubbling profiles. 

We will assume the following hypotheses for the nonlinearity 
$$f(t) = t^{2^*_s-1}\,L(t).$$

\begin{enumerate}[label=\textbf{(f\arabic*)$_\infty$}]
	\item
	\label{f1}
	The nonlinearity $f$  is {\it slightly subcritical} at infinity, in the sense that
	$${\D\lim_{t\to +\infty}}\,\frac{f(t)}{t^{\frac{N+2s}{N-2s}}}=0.$$
	
	\item 
	\label{f2}
	The nonlinearity $f$ is {\it slightly superlinear} at infinity in the sense of de Figueiredo, Lions and Nussbaum (see \cite[p. 43]{deF_Lion_Nus_1982}). That is,
	\begin{equation*}
		{\D\liminf_{t \to +\infty}}\, \frac{f(t)}{t}> \la_{1,s},
	\end{equation*}
	where $\la_{1,s}$ is the first eigenvalue of the fractional Laplacian in $\Om$ with homogeneous exterior Dirichlet condition.
	\smallskip 
	
	\item
	\label{f3}  The function $L$ is 
	$C^1(0,+\infty),$  and $L'$  satisfies:
	\begin{align*}
		&L'(t)< 0,\quad \hbox{ for all }  t\ge t_1,  \hbox{ with some } t_1\gg 1,\label{L':neg}\\
        \text{ and } \\
		&|L'|\in RV_{-1}, \text{ that is,} \ {\D\lim_{t\to +\infty}}\, \frac{|L'(\tau t)|}{|L'(t)|} =\frac{1}{\tau}\qquad\forall  \tau >0.\nonumber
	\end{align*}
\end{enumerate}

Our main result shows that the growth condition on the slowly varying part $L$ acts as an asymptotic friction that prevents the concentration of solutions.  

\begin{thm}\label{th:apriori:intro}
    Given a bounded $C^{1,1}$ convex domain $\Omega \subset \mathbb{R}^N$, let $s \in (0,1)$ and $N > 2s$.  Assume that the nonlinearity $f(t) = t^{\frac{N+2s}{N-2s}} L(t)$ satisfies hypotheses {\rm \ref{f1}--\ref{f3}} and the structural condition:
	\begin{equation}\label{hypothesis1:intro}
		\lim_{t \to +\infty} \frac{t \, |L'(t)|}{L^{\frac{N}{2s}}(t)} = +\infty.
	\end{equation}
	Then, there exists a constant $C > 0$, depending only on $N, s, \Omega$, and $f$, such that any positive weak solution $u \in H_0^s(\Omega)$ to \eqref{eq:ell:pb} satisfies
	\begin{equation*}
		\|u\|_{L^\infty(\Omega)} \le C.
	\end{equation*}
\end{thm}

Note that the condition~\eqref{hypothesis1:intro}  is equivalent to
$$
\frac{t^{\frac{N}{2s}+1}\,\big|tf'(t)-qf(t)\big|}{f(t)^{\frac{N}{2s}}}  \to +\infty  \qq{as} t\to\infty,
$$
with $q=\frac{N+2s}{N-2s}$. Indeed, since $f(t) = t^{q}\,L(t)$, we obtain by differentiating
$
f'(t)=qt^{q-1}\,  L(t) +t^{q}\,L'(t).
$
Hence, it follows
\begin{align*}
\frac{t^{\frac{N}{2s}+1}\, \big(tf'(t)-qf(t)\big)}{f(t)^{\frac{N}{2s}}}
=\frac{t^{\frac{N}{2s}+1+q}\,    \big(t\,L'(t)\big) }{t^{q\frac{N}{2s}}\, L(t)^{\frac{N}{2s}}} 
= \frac{t\,L'(t)}{L(t)^{\frac{N}{2s}}}.
\end{align*}

Let us also notice that due to the monotonicity in the fractional parameter $s$, hypothesis~\eqref{hypothesis1:intro} is always satisfied in the non local framework whenever it is satisfied in the local one ($s=1$).

The proof of the theorem proceeds in three main steps. First, we combine the uniform boundary estimates for positive solutions of \eqref{eq:ell:pb} established in \cite{RosOton_Serra_extremal} (see also Theorem~\ref{th:bdd}) with the fractional Pohozaev identity \cite{RosOton_Serra_poh} (cf. Theorem~\ref{Pohozaev}). This allows us to prove that the integral quantity
\begin{equation} \label{H1:bd}
    \int_\Omega u^{2^*_s+1}|L'(u) |\,dx \sim \int_\Omega \big( 2^*_s \, F(u)-u\,f(u) \big)\,dx
\end{equation}
is uniformly bounded. Second, by exploiting fractional Morrey embeddings, elliptic regularity theory, and the asymptotic properties of slowly varying functions, we obtain a lower bound for the radius of a ball in which the solution exceeds half of its $L^\infty(\Om)$ norm (see \cite{Castro_Pardo_RMC_2015, Cuesta_Pardo_MedJM} for the analogous local technique). Finally, using the integral bound~\eqref{H1:bd} and localizing it within this ball, we deduce that
\begin{equation*}
\frac{\|u\|_\infty \, \big|L'\big(\|u\|_\infty\big)\big|}{L^\frac{N}{2s}\big(\|u\|_\infty\big)} \le C < +\infty, 
\end{equation*}
which, in conjunction with hypothesis~\eqref{hypothesis1:intro}, yields the uniform boundedness of the solutions.

In \cite[Table 1]{Cuesta_Pardo_MedJM} can be found a list of examples of slowly varying functions for the case $q=2^*-1$ satisfying our hypothesis \ref{f1}-\ref{f3}.  It can be easily checked that in cases 1 and 2 concerning 
$$
L_1(s):=\left[\log (K + s)\right]^\al , 
$$
and
$$
L_2(s):=\Big[ \log (K+s)/\log\big(K +\log (K+s)\big)\Big]^\al,
$$
both for $\al< 0,\ K > 1 ,$
hypothesis~\eqref{hypothesis1:intro}  is satisfied whenever the parameter $|\al|>\frac{2s}{N-2s}.$\\
\medskip

For the sake of completeness, we also establish the uniform boundedness of solutions in the strictly subcritical regime. In this setting, thanks to the gap to the critical fractional Sobolev exponent, one obtains a uniform $H^s(\mathbb{R}^N)$ estimate without imposing any additional structural or monotonicity assumptions.

\begin{thm}\label{th:aprioriSubcritical:intro}
    Given a bounded $C^{1,1}$ convex domain $\Omega \subset \mathbb{R}^N$, let $s \in (0,1)$ and $N > 2s$.  Assume that $q \in \big[1,\frac{N+2s}{N-2s}\big)$ and the nonlinearity $f(t) = t^{q} L(t)$ satisfies hypotheses {\rm \ref{f1}--\ref{f2}} with $L\in RV_0$.
    
    Then, there exists a constant $C > 0$, depending only on $N, s, \Omega$, and~$f$, such that any positive weak solution $u \in H_0^s(\Omega)$ to \eqref{eq:ell:pb} satisfies
    \begin{equation*}
        \|u\|_{L^\infty(\Omega)} \le C.
    \end{equation*}
\end{thm}

This paper is organized as follows. In Section~\ref{sec:prel} we include some preliminaries on linear and nonlinear fractional elliptic equations, as well as functions of regular variation. Section~\ref{tec:4} contains the proof of our main results, Theorems~\ref{th:apriori:intro} and \ref{th:aprioriSubcritical:intro}.

\section{Preliminaries}
\label{sec:prel}

\subsection{Fractional Sobolev inequalities}

$$ \\ $$
\vspace{-13mm}

The analysis of fractional elliptic equations requires a precise understanding of the spaces where the solutions lie. In particular, the fractional Sobolev inequalities are one of the main ingredients to establish uniform a priori estimates, as they provide the necessary bridge between the energy functional and the integrability properties of the solutions.

Throughout this work, given a domain $\Omega \subset \mathbb{R}^N$, $s \in (0,1)$ and $q \in [1, \infty)$, we define the fractional Sobolev space $W^{s,q}(\Omega)$ as:
$$ 
W^{s,q}(\Om):=\left\{ u\in L^{q}(\Om)\ :\
\frac{|u(x)-u(y)|}{|x-y|^{N/q+s}}
\in L^{q}(\Om\times\Om)\right\},
$$
which is a Banach space under the norm
$$
\|u\|_{W^{s,q}(\Om)}:=\left(\|u\|_{L^{q}(\Om)}
+\int_\Om\int_\Om
\frac{|u (x)-u (y)|^q}{|x-y|^{N+sq}}\
dx dy
\right)^{1/q}.
$$

For $s \geq 1$ not being an integer, let $s = k + \sigma$ where $k = \lfloor s \rfloor$ is its integer part and $\sigma = \{s\} \in (0,1)$ is its fractional part. The fractional Sobolev space $W^{s,q}(\Omega)$ is defined as
$$ W^{s,q}(\Omega) := \left\{ u \in W^{k,q}(\Omega) : D^\alpha u \in W^{\sigma,q}(\Omega) \text{ for all } |\alpha| = k \right\},$$
endowed with the norm
$$ \|u\|_{W^{s,q}(\Omega)} := \|u\|_{W^{k-1,q}(\Omega)} + \sum_{|\alpha|=k} \|D^\alpha u\|_{W^{\sigma,q}(\Omega)}.$$

Analogously to the classical case ($s=1$), these spaces satisfy embedding properties where functions gain integrability by balancing the fractional order of differentiation. The following result summarizes the main embeddings:

\begin{thm}[Theorems 5.4, 6.7, 6.10, and 8.2 in\cite{Nezza_Palatucci_Valdinoci}]
	\label{th:fr:Sobolev:emb}
Let $\Om\subset \R^N$ be an open Lipschitz domain with bounded boundary. Then,
\begin{enumerate}
	\item[(i)] Subcritical case: if $sq < N$, there exists $C=C(N,q,s,\Omega) > 0$ such that for any $u \in W^{s,q}(\Omega)$, it satisfies
	$$ \|u\|_{L^{p}(\Omega)} \le C \|u\|_{W^{s,q}(\Omega)}, \quad \text{for any } p \in [q, q^*_s], $$
	where 
	$$q^*_s := \frac{Nq}{N-sq}$$
	is the fractional critical exponent. Moreover, in the case $p=q^*_s$, the constant $C$ is independent of the domain $\Omega$.
	
	\item[(ii)] Critical case: if $sq=N$, there exists $C=C(N,q,s,\Om)>0$ such that, for any $u \in W^{s,q}(\Om)$, it follows
	$$ \|u\|_{L^{p}(\Om)}\le C  \|u\|_{W^{s,q}(\Om)}, \qq{for any} p\in [q,\infty).$$
	
	\item[(iii)] Supercritical case: if $sq > N$, functions in $W^{s,q}(\Omega)$ are continuous. Furthermore,
	\begin{equation*}
		\|u\|_{C^{0,\alpha}(\Omega)} \le C \|u\|_{W^{s,q}(\Omega)}, 
	\end{equation*}
	with $\alpha = s - N/q$ and $C=C(N,q,s,\Om)>0$.
\end{enumerate}
\end{thm}

While the previous theorem focuses on integrability, it is often necessary to track how functions move between different fractional scales. The following result, concerning nested embeddings, will be fundamental for our bootstrap iterations.

\begin{thm}[Theorem A.2. in \cite{Zuazua1}] \label{th:fr:Sobolev:emb2}
	Given $\Om\subset \R^N$ a bounded Lipschitz domain, let $0<s\leq r$ and $1<q\leq p < \infty$ be real numbers such that
	$$ r-\frac{N}{q} \geq s-\frac{N}{p}.  $$
	Then, for any $u \in W^{r,q}(\Om)$, we have
	\begin{equation*}
		\|u\|_{W^{s,p}(\Om)} \le C \, \|u\|_{W^{r,q}(\Om)},
	\end{equation*}
	where $C>0$ is a constant depending only on $N$, $s$, $r$, $p$, $q$ and $\Omega$.
\end{thm}

\subsection{Regularity for the linear fractional Dirichlet problem}

$$ \\ $$
\vspace{-13mm}

In this subsection, we recall the regularity theory for the fractional linear problem, which differs significantly from the classical case near the boundary. Let us consider
\begin{equation}\label{eq:linear}
	\begin{cases}
		(-\Delta)^s \,u = g, &\qquad \text{in }\, \Omega,  \\
		\phantom{(-\Delta)^s} \, u = 0, &\qquad  \text{in }\, \mathbb{R}^N \setminus \Omega.
	\end{cases}
\end{equation}

The natural energy space for this problem is $H^s_0(\Omega)$, which consists of functions in $H^s(\mathbb{R}^N):=W^{s,2}(\R^n)$ that vanish outside $\Omega$. Then, it follows by standard arguments that if $g \in (H^s_0(\Omega))^*$, there exists a unique weak solution $u \in H^{s}_0(\Omega)$.

When the right-hand side $g$ possesses higher integrability, the solution gains differentiability in the interior of the domain, following the fractional analogue of the Calderón-Zygmund estimates.

\begin{thm}[Theorem 3. in \cite{Zuazua2}] \label{th:local:reg}
Given $\Omega \subset \R^N$ a bounded open set, let $g\in \left(H^s_0(\Omega)\right)^* \cap L^{q}(\Om)$ with $q\in [2,\infty)$ and let $u \in H^{s}_0(\Om)$ be the unique weak solution to the Dirichlet problem \eqref{eq:linear}. Then, for any $\widetilde{\Omega} \subset \subset \Omega$ there exists a positive constant $C=C(N,q,s,\Omega,\widetilde{\Omega})$ such that the solution of the problem satisfies
$$\|u\|_{W^{2s,q}(\widetilde{\Omega})}\le C \, \|g\|_{L^{q}(\Om)}.  $$
\end{thm}

A crucial observation is that, unlike the classical Laplacian ($s=1$), the solution to \eqref{eq:linear} with $g \in L^\infty(\Omega)$ is not necessarily $W^{2s,q}$ or $C^{1, \alpha}$ up to the boundary. Instead, solutions behave like $\text{dist}(x, \partial \Omega)^s$ near $\partial \Omega$. The following result provides the optimal global regularity and characterizes this boundary behavior.

\begin{thm}[Proposition 1.1 and Theorem 1.2 in \cite{RosOton_Serra_reg}]\label{th:GlobalBoundaryRegularity}
	Let $\Om$ be a  bounded, $C^{1,1}$ domain and  $g \in L^\infty(\Om )$. Let $u$ be the solution of the fractional problem \eqref{eq:linear}. Then, $u \in C^s(\mathbb{R}^N)$ and
	\begin{equation*}\label{eq:reg}
		\|u\| _{C^{s} (\R^N )} \leq C\, \|g \| _ {L^\infty (\Om ) },
	\end{equation*}
	where $C$ depends only on $\Om$ and $s$.

	Furthermore, let $\delta(x) = \text{dist}(x, \partial \Omega)$. Then, the function $u/\delta^s$ can be continuously extended to $\overline{\Omega}$ and satisfies
	\begin{equation*}\label{eq:reg:2}
		\left\|\frac{u}{\de^s}\right\| _{C^{\al}(\overline{\Omega})} \leq C\, \|g \| _ {L^\infty (\Om ) },
	\end{equation*}
	for some $0<\al<\min\{s,1-s\}$. The constants $\al$ and $C$, depend only on $\Om$ and $s$.	
\end{thm}

Note that the Hölder regularity of $u/\delta^s$ up to the boundary is a cornerstone in the derivation of the fractional Pohozaev identity, as it ensures that the nonlocal normal derivatives are well-defined.

\subsection{Uniform bounds in a neighborhood of the boundary for the nonlinear problem.}

$$ \\ $$
\vspace{-10mm}

In this subsection, we focus on the behavior of solutions near the boundary $\partial \Omega$. A key step in obtaining global a priori estimates is to show that any potential blow-up of the solutions is localized in the interior of the domain. We begin by recalling a result by Ros-Oton and Serra regarding the boundary behavior of positive solutions.

\begin{pro}[Proposition 1.8 and Lemma 6.1 in \cite{RosOton_Serra_extremal}]  \label{lem:1}
Given a bounded convex  domain $\Om \subset {\mathbb R} ^N $, $s\in (0, 1)$ and a locally Lipschitz function $f$, let $u$ be a bounded positive solution of the semilinear problem~\eqref{eq:ell:pb}.

Then, there exist constants $d_0>0$ and $C_0>0,$ depending only on the domain $\Om$, such that
\begin{equation*}\label{bdd:4}
\|u\|_{L^\infty(\Om\setminus \Om_{d_0})}\leq C_0 \, \|u\|_{L^1(\Om)},
\end{equation*}
and
\begin{equation}\label{bdd:5}
	\|u/\delta^s\|_{L^\infty(\Om\setminus \Om_{d_0})}\leq C_0 \left(\|u\|_{L^1(\Om)}+ \|f(u)\|_{L^\infty(\Om\setminus \Om_{d_0})} \right),
\end{equation}
where $\delta(x) = dist(x,\p\Om)$ and $\Om_d=\{x\in\Om \ \text{ such that } \ \delta(x)>d\}$.
\end{pro}

By combining these boundary estimates with a classic $L^1$-control via the first eigenfunction, we can establish uniform a priori bounds near the boundary for slightly superlinear non-linearities.

\begin{thm}\label{th:bdd}
Assume $\Omega \subset \mathbb{R}^N$ is a bounded, convex domain with $C^{1,1}$ boundary. Let $s\in (0, 1)$, $f$ be a locally Lipschitz function satisfying the superlinearity condition \ref{f2}, and $u$ be a bounded positive solution of the problem \eqref{eq:ell:pb}.

Then, there exist constants $d_1>0$ and $C_1>0,$ depending only on the domain $\Om$ and the nonlinearity $f$, such that
$$ \|u/\delta^s\|_{L^\infty(\Om\setminus \Om_{d_1})}\leq C_1. $$
\end{thm}

\begin{proof}
	
Let $u$ be a positive weak solution of \eqref{eq:ell:pb}. First, we show that the $L^1$ norm of $u$ is bounded. Following the method in \cite{Turner, deF_Lion_Nus_1982}, the superlinearity of $f$ implies that the solutions satisfy an $L^1$ estimate against the first eigenfunction $\phi_{1,s}$. Let $\lambda_{1,s}$ be the first eigenvalue of $(-\Delta)^s$ in $\Omega$ with zero Dirichlet conditions. By the self-adjointness of the operator:
\begin{equation*}
	\lambda_{1,s} \int_{\Omega} u \phi_{1,s} = \int_{\Omega} (-\Delta)^s u \phi_{1,s} = \int_{\Omega} f(u) \phi_{1,s}.
\end{equation*}
From the superlinearity condition \ref{f2}, for any $\varepsilon > 0$, there exists $C_\varepsilon\geq 0$ such that $f(t) \geq (\lambda_{1,s} + \varepsilon)t - C_\varepsilon$. Choosing $\varepsilon > 0$, we obtain
\begin{equation*}\label{cota:f:L13}
\la_{1,s}\int_{\Om}u\,\phi_{1,s}  = \int_{\Om}f(u)\,\phi_{1,s} \ge (\la_{1,s}+\varepsilon)\int_{\Om}u\,\phi_{1,s} -C_\varepsilon,
\end{equation*}
which leads to the uniform estimate:
\begin{equation}\label{cota:f:L10}
	\int_{\Omega} u \phi_{1,s} \leq \widetilde{C}.
\end{equation}

Since $\phi_{1,s} > 0$ in $\Omega$ (by the Strong Maximum Principle for the fractional Laplacian \cite{Silvestre}), and the set $\overline{\Omega_{d_1}}$ is a compact subset of $\Omega$, we have $m_{d_1} := \min_{\Omega_{d_1}} \phi_{1,s} > 0$. Thus, \eqref{cota:f:L10} yields
\begin{equation*}\label{int_u_mid}
	\int_{\Omega_{d_1}} u \leq \frac{1}{m_{d_1}} \int_{\Omega_{d_1}} u \phi_{1,s} \leq \frac{\widetilde{C}}{m_{d_1}} =: \widetilde{C}_1.
\end{equation*}

Next, we use an absorption argument. Let $0<d_1 \leq d_0$ be small enough such that $C_0 \, |\Omega \setminus \Omega_{d_1}| \leq 1/2$, where $C_0$ is the constant from Proposition~\ref{lem:1}. Splitting the $L^1$ norm and using the $L^\infty(\Om)$ bound near the boundary, it follows
\begin{align*}
	\|u\|_{L^\infty(\Omega \setminus \Omega_{d_1})} &\leq C_0 \left( \int_{\Omega_{d_1}} u + \int_{\Omega \setminus \Omega_{d_1}} u \right) \\
	&\leq C_0 \widetilde{C}_1 + C_0 |\Omega \setminus \Omega_{d_1}| \|u\|_{L^\infty(\Omega \setminus \Omega_{d_1})} \\
	&\leq C_0 \widetilde{C}_1 + \frac{1}{2} \|u\|_{L^\infty(\Omega \setminus \Omega_{d_1})}.
\end{align*}
This gives 
$$\|u\|_{L^\infty(\Omega \setminus \Omega_{d_1})} \leq 2 C_0 \widetilde{C}_1.$$ 
Consequently, the total $L^1$ norm is also bounded by
$$\|u\|_{L^1(\Omega)} \leq \widetilde{C}_1 + 2 C_0 \widetilde{C}_1 |\Omega \setminus \Omega_{d_1}| \leq 2\,\widetilde{C}_1.$$

Finally, applying \eqref{bdd:5} from Proposition \ref{lem:1} and using the fact that the nonlinearity $f$ is locally Lipschitz (and thus bounded on bounded sets), we conclude
\begin{align*}
	\|u/\delta^s\|_{L^\infty(\Om\setminus \Om_{d_1})}
	&\leq C_0 \left(\|u\|_{L^1(\Om)}+ \|f(u)\|_{L^\infty(\Om\setminus \Om_{d_0})} \right)\\
	&\leq 2\,C_0 \, \widetilde{C}_1+ C_0\, \|f\|_{L^\infty(0,\|u\|_{L^\infty(\Om\setminus \Om_{d_0})})} \\
	&\leq 2\,C_0 \, \widetilde{C}_1+ C_0\,\|f\|_{L^\infty(0,2\,C_0 \, \widetilde{C}_1)} =:C_1.
\end{align*}
\end{proof}

\subsection{Pohozaev Identity for the fractional Laplacian.}

$$ \\ $$
\vspace{-13mm}

A fundamental tool in the study of semilinear elliptic equations, both in the classical and fractional frameworks, is the Pohozaev identity. This integral relation provides a global balance between the interior energy of the solution and its behavior near the boundary. In the fractional setting, this identity was established by Ros-Oton and Serra in \cite{RosOton_Serra_poh}. 

A striking feature of the fractional Pohozaev identity is that the role of the classical normal derivative $\partial u / \partial \nu$ is played by the nonlocal analogue $u/\delta^s|_{\partial \Omega}$, that was also present in the regularity results from the previous subsections. 

\begin{thm}[Theorem 1.1. in \cite{RosOton_Serra_poh}]  \label{Pohozaev}
Given a bounded $C^{1,1}$ domain $\Om \subset {\mathbb R} ^N $, $s\in (0, 1)$, and a locally Lipschitz function $f$, let $u$ be a bounded solution of the semilinear problem~\eqref{eq:ell:pb}.
Then,
$$ 
u/\de^s \in C^\alpha(\overline{\Omega}),\qq{for some} \alpha\in (0, 1),$$
meaning that $u/\de^s$ has a continuous extension up to the boundary $\partial \Om$ which is $C^\alpha(\overline{\Om})$, and the
following identity holds
\begin{equation*}\label{Poho:Id}
	2N \int_{\Omega} F(u) \, dx - (N-2s) \int_{\Omega} f(u)u \, dx = \Gamma(1+s)^2 \int_{\partial\Omega} \left(\frac{u}{\delta^s}\right)^2 (x \cdot \nu) \, d\sigma,
\end{equation*}
where $F(t) := \int _0 ^t f(r) \, dr $ is a primitive of the function $f$,
$\nu (x)$ is the unit outward normal to $\p\Om$ at $x$, and $\Ga$ is the Gamma function.	
\end{thm}

The importance of this identity lies in its ability to detect the critical growth of the nonlinearity. Indeed, for the critical power $f(u) = u^{2^*_s-1}$, the coefficients on the left-hand side satisfy $2N/2^*_s = N-2s$, which causes the interior terms to cancel out in a specific way, revealing the obstruction to existence in star-shaped domains. In our context, this identity will be used to derive the uniform $L^\infty(\Om)$ bounds via a contradiction argument.

\subsection{Some properties of regularly varying functions}
\label{subsection:rvf}
$$ \\ $$
\vspace{-13mm}

Trivially,  for  any function $g\in RV_q$ there exists $L\in RV_0$ such that   $g(t)=t^q L(t)$ for $t>0$ large. 
\medskip

The following result was proved by Karamata \cite{Karamata} in the continuous case and extended by Korevaar, van Aardenne-Ehrenfest and de Bruijn in the measurable case, see \cite{Korevaar_Aardenne-Ehrenfest_Bruijn}, for the following result see 
\cite[Theorem 1.1]{Seneta}  and \cite[Theorem 1.5.2]{Bingham}.

\begin{thm}[Uniform Convergence Theorem]
\label{usfF}
Let $g\in C(0,+\infty)$ positive at infinity and $g\in RV_q$ for some $q\in\R$.  Let $G(t)=:\int_0^s g(t)\,dt$. 
Then: 
\begin{enumerate}[label=\rm{(\roman*)}]
\item  $G$ is a function regularly varying at infinity of index $q+1$,
and satisfies
\begin{align} 
\lim_{t\to+\infty}\frac{tg(t)}{G(t)}=q+1.
\label{eq:usfF}
\end{align}
\item  Furthermore, for all $a,  b\in \R^+,$ $0<a<b$,  the limit  \eqref{reva}  
is uniform for $ \tau \in [a,b].$
\item  
Moreover, if $g$ is bounded close to 0 and $q>0$,  then
the limit  \eqref{reva} is uniform for $ \tau \in (0,b]$, for any $b>0$. 
\end{enumerate} 
\end{thm}

\section{Proof of the main results}
\label{tec:4}

We recall that Theorem~\ref{th:apriori:intro} addresses the critical case $q = \frac{N+2s}{N-2s}$. Under assumption \ref{f1}, this slightly subcritical behavior implies that $L(t) \to 0$ as $t \to \infty$. From now on, throughout this section, $C$ will denote various generic positive constants that are independent of the solution $u$.

\begin{lem}\label{lem:3.1}
Assume that $f$ satisfies hypotheses {\rm \ref{f1}--\ref{f3}}, and let $L$ be the slowly varying function in the sense of Definition~\ref{Def:RVandSV}. Then, there exists a constant $C > 0$ such that for any positive weak solution $u \in H_0^{s}(\Omega)$ to \eqref{eq:ell:pb}, the following uniform integral bound holds:
\begin{equation}\label{u:fu:C}
\int_\Omega u^{\frac{2N}{N-2s}+1} |L'(u)| \, dx \le C. 
\end{equation}
\end{lem}

\begin{proof}
We split the proof into four steps. First, we obtain a uniform boundary estimate. Next, we use this boundary control within the fractional Pohozaev identity. Then, we establish an asymptotic equivalence for the energy functional, and finally, we conclude the uniform integral estimate.

\noindent {\bf Step 1}. {\it Boundary uniform estimates}.

By Theorem \ref{th:bdd}, we know that there exist constants $d_1>0$ and $C_1 > 0$, depending only on $f$ and $\Omega$, such that
$$\|u/\delta^s\|_{L^\infty(\Omega \setminus \Omega_{d_1})} \leq C_1.$$
Furthermore, Theorem~\ref{Pohozaev} ensures that the function $u/\delta^s$ admits a continuous extension up to the boundary $\partial \Omega$. Consequently, the boundary values satisfy
$$\|u/\delta^s\|_{L^\infty(\partial \Omega)} \leq C_1.$$

\medskip

\noindent{\bf Step 2}. {\it  A Pohozaev's type estimate}.

Applying the fractional Pohozaev identity (Theorem~\ref{Pohozaev}), any positive solution $u$ of \eqref{eq:ell:pb} satisfies
$$ \frac{2N}{N-2s} \int_\Omega F(u) \, dx - \int_\Omega u f(u) \, dx = \frac{\Gamma(1 + s)^2}{N-2s} \int_{\partial \Omega} \left(\frac{u}{\delta(x)^s}\right)^2 (x \cdot \nu) \, d\sigma, $$
where $\nu(x)$ denotes the unit outward normal to $\partial \Omega$ at $x$. Since the right-hand side is uniformly bounded independently of $u$ thanks to the boundary control from Step 1, we obtain
$$	\left| \frac{2N}{N-2s} \int_\Omega F(u) \, dx - \int_\Omega u f(u) \, dx \right| \le C. $$

\medskip

\noindent \textbf{Step 3.} \textit{Asymptotic equivalence at infinity.}

We claim that
$$ \lim_{t \to +\infty} \frac{\frac{2N}{N-2s} F(t) - t f(t)}{t^{\frac{2N}{N-2s}+1} |L'(t)|} = \frac{N-2s}{2N}. $$
Indeed, one can check that
$$ \left( \frac{2N}{N-2s} F(t) - t f(t) \right)' = -t^{\frac{2N}{N-2s}}\,L'(t), $$
which gives
$$ \lim_{t \to +\infty} \frac{\frac{2N}{N-2s} F(t) - t f(t)}{t^{\frac{2N}{N-2s}+1} |L'(t)|} = \lim_{t \to +\infty} \frac{-\int_0^t \tau^{\frac{2N}{N-2s}} L'(\tau) \, d\tau}{t^{\frac{2N}{N-2s}+1} |L'(t)|}. $$

To evaluate this limit, let us define 
$$g(t) := t^{\frac{2N}{N-2s}} |L'(t)| \quad \text{ and its primitive } \quad G(t) := \int_0^t g(\tau) \, d\tau.$$
Since $|L'| \in RV_{-1}$ by hypothesis \ref{f3}, one can check directly by the definition of regularly varying functions \eqref{reva} that $g \in RV_{\frac{N+2s}{N-2s}}$. By Karamata's Theorem for integrals (see Theorem 1.5.11. in \cite{Bingham}), we deduce that $G \in RV_{\frac{2N}{N-2s}}$ and
$$\lim_{t \to +\infty} \frac{t \, g(t)}{G(t)} = \frac{2N}{N-2s},$$
which explicitly translates to
$$	\lim_{t \to +\infty} \frac{\int_0^t \tau^{\frac{2N}{N-2s}} |L'(\tau)| \, d\tau}{t^{\frac{2N}{N-2s}+1} |L'(t)|} = \lim_{t \to +\infty} \frac{G(t)}{t \, g(t)} =  \frac{N-2s}{2N}. $$
Finally, recalling from \ref{f3} that $L'(t) < 0$ for all $t \ge t_1$, we conclude the claim.

\bigskip

\noindent \textbf{Step 4.} \textit{A uniform estimate conclusion.}
\smallskip

On the one hand, from the asymptotic equivalence established in Step 3, there exists $T > 0$ large enough such that for all $t > T$, it follows
$$ \frac{2N}{N-2s} F(t) - t f(t) \ge \frac{1}{2} \left( \frac{N-2s}{2N} \right) t^{\frac{2N}{N-2s}+1} |L'(t)|.$$

On the other hand, since both $f$ and its primitive $F$ are bounded on the compact set $[0, T]$ we have 
$$ \frac{2N}{N-2s} F(t) - t f(t) \ge -C$$
for all $t\in [0,T]$.

Hence, integrating over the domain $\Omega$, which we split into regions where $u > T$ and $u \le T$, we obtain
$$
\frac{2N}{N-2s} \int_\Omega F(u) \, dx - \int_\Omega u f(u) \, dx \ge \frac{N-2s}{4N} \int_\Omega u^{\frac{2N}{N-2s}+1} |L'(u)| \, dx - C,
$$
for some constant $C > 0$ independent of $u$. Combining this lower bound with the Pohozaev upper bound derived in Step 2 directly yields the conclusion of the result.

\end{proof}

With the uniform integral estimate of Lemma \ref{lem:3.1} at hand, we are now ready to prove our main result. The core mechanism of the forthcoming proof relies on translating this global integral bound into a pointwise $L^\infty(\Omega)$ estimate. To achieve this, we will assume by contradiction that the supremum of a sequence of solutions blows up, and we will carefully track how the volume of the regions where the solutions peak behaves. This local analysis will allow us to evaluate the integral bound from below, ultimately forcing a clash with the structural hypothesis \eqref{hypothesis1:intro}.

\begin{proof}[Proof of Theorem \ref{th:apriori:intro}]
	
Before detailing the estimates, let us outline the strategy of the proof. We will estimate from below the radius of a ball where a solution exceeds half of its $L^\infty(\Omega)$ norm. This technique has been successfully employed in the context of the classical Laplacian for single equations and systems, as well as for the $p$-Laplacian (see, for instance, \cite{Castro_Pardo_RMC_2015, Cuesta_Pardo_MedJM, Mavinga_Pardo_JMAA, Damascelli_Pardo}). 

To achieve this lower bound in our nonlocal setting, we will combine fractional elliptic regularity theory in $W^{2s,q}$ for $q \in \left( \frac{N}{2s}, \frac{N}{s} \right) \cap [2,+\infty)$, fractional Sobolev and Morrey embeddings, and the asymptotic properties of slowly varying functions. 

By feeding this radius estimate back into the integral bound obtained in Lemma \ref{lem:3.1}, we will show that if there exists a sequence of positive weak solutions $\{u_n\}_{n \in \mathbb{N}} \subset H_0^{s}(\Omega)$ to \eqref{eq:ell:pb} such that $\|u_n\|_\infty \to +\infty$, then they must satisfy the following asymptotic behavior:
$$ \frac{\|u_n\|_\infty \, \big|L'(\|u_n\|_\infty)\big|}{L^{\frac{N}{2s}}(\|u_n\|_\infty)} \le C < +\infty,$$
for some constant $C$ independent of $n$. This uniform bound clearly contradicts the superlinear growth assumption \eqref{hypothesis1:intro}, implying that the $L^\infty(\Om)$ norm of all positive solutions must be uniformly bounded, thereby proving our result.

Let us assume by contradiction that the result is not true. Then, there exist $\{u_n\}_{n\in \mathbb{N}} \subset H_0^{s}(\Omega)$, a sequence of positive weak solutions to \eqref{eq:ell:pb} such that $\|u_n\|_\infty \to +\infty$. Without loss of generality, we can assume that 
$$ \|u_n\|_\infty \ge \max\{t_1, 3\,C_1\,d_1^s\}, $$
where $t_1 > 0$ is the threshold from hypothesis \ref{f3} and $C_1, d_1$ are the constants from the boundary estimate in Theorem \ref{th:bdd}.

By the monotonicity condition $L'(t) < 0$ for $t > t_1$ coming from hypothesis \ref{f3}, the term $|L'(u_n)|^{-1}$ is well-defined on the open set $\{x \in \Omega : u_n(x) > t_1\}$. Let us take any $q \in \left( \frac{N}{2s}, \frac{N}{s} \right) \cap [2,+\infty)$. Note that we can always do it by the assumption $N>2s$. In order to estimate the $L^q$ norm of the nonlinearity, we decompose the integral as
$$ \int_{\Omega} |f(u_n)|^q \, dx = \int_{\Omega \cap \{u_n > t_1\}} |f(u_n)|^q \, dx + \int_{\Omega \cap \{u_n \leq t_1\}} |f(u_n)|^q \, dx. $$
For the first term, we can rewrite the integrand to exploit the integral bound from Lemma \ref{lem:3.1} as
\begin{align*}
	|f(u_n)|^q  &= \left( u_n^{\frac{N+2s}{N-2s}q - \left(\frac{2N}{N-2s} + 1\right)} L(u_n)^q |L'(u_n)|^{-1} \right) u_n^{\frac{2N}{N-2s} + 1} |L'(u_n)|.
\end{align*}
Then, let us define the auxiliary function for $t > t_1$:
$$ h(t) := t^{\frac{(N+2s)q - (2N + (N-2s))}{N-2s}} L(t)^q \, |L'(t)|^{-1}.$$
Since $|L'| \in RV_{-1}$ and $q > \frac{N}{2s}$, a direct computation on regular variation shows that $h \in RV_{\gamma}$ with index 
$$ \gamma = \frac{q(N+2s) - 2N}{N-2s} > 0. $$
Since $\gamma > 0$, we can apply the Uniform Convergence Theorem for regularly varying functions (see Theorem~1.5.2. in \cite{Bingham} and Corollary~3.5. in \cite{Cuesta_Pardo_MedJM}) to get
$$ \|h\|_{L^\infty([0,T])} \leq C \,  h(T) \quad \text{ for each } \ T\geq t_1. $$
Hence, using the uniform integral bound \eqref{u:fu:C} from Lemma \ref{lem:3.1}, we obtain
\begin{align*}
	\int_{\Omega} |f(u_n)|^q &\leq |\Omega| \|f\|_{L^\infty([0,t_1])}^q + \|h\|_{L^\infty([0,\|u_n\|_\infty])} \int_{\Omega} u_n^{\frac{2N}{N-2s}+1} |L'(u_n)| \\
	&\leq C \left( 1 + h(\|u_n\|_\infty) \right) \leq C h(\|u_n\|_\infty),
\end{align*}
where the last inequality holds for $n$ sufficiently large since $h(t) \to \infty$ when $t \to \infty$ as follows from Proposition~1.3.6. in \cite{Bingham}.

By the local elliptic regularity estimate (Theorem~\ref{th:local:reg}) and the fractional Sobolev embeddings (Theorem~\ref{th:fr:Sobolev:emb2}), we have
$$ \|u_n\|_{W^{s,q^*_s}(\Omega_d)} \le C \|u_n\|_{W^{2s,q}(\Omega_d)} \le C \|f(u_n)\|_{L^{q}(\Omega)} \le C h(\|u_n\|_\infty)^{1/q},$$
where we have crucially used the fact that $q\geq 2$.

Next, since we choose $q \in \left( \frac{N}{2s}, \frac{N}{s} \right)$, the fractional Sobolev exponent is such that $q^*_s = \frac{Nq}{N-qs} > N$. Therefore, applying Morrey embedding (Theorem \ref{th:fr:Sobolev:emb}.(iii)), it follows that $u_n$ is H\"older continuous in $\Omega_d$ with exponent $\alpha = s - N/q^*_s = 2s - N/q \in (0,1)$. In addition, there exists a constant $C > 0$ such that
\begin{equation*}\label{Morrey}
	|u_n(x) - u_n(y)| \le C \|u_n\|_{W^{s,q^*_s}(\Omega_d)} |x-y|^{2s - N/q}, \quad \forall x, y \in \Omega_d.
\end{equation*}

Now, let $x_n \in \Omega$ be a point where the maximum of the function~$u_n$ is attained. That is,
 $u_n(x_n) = \|u_n\|_\infty$. By the boundary estimate (Theorem \ref{th:bdd}) and the choice of $\|u_n\|_\infty \ge 3 \,C_1 d_1^s$, we know $x_n \in \Omega_d$. We define the ``peak-ball'' $B_n := B(x_n, R_n) \subset \Omega$, where the radius $R_n$ is chosen such that $u_n(x) \ge \frac{1}{2} \|u_n\|_\infty$ for all $x \in B_n$, and there exists a point $y_n \in \partial B_n$ such that $u_n(y_n) = \frac{1}{2} \|u_n\|_\infty$. Note that this is always possible due to the continuity of the solutions. In addition, $y_n\in \Omega_d$ too.

Evaluating the Hölder estimate at the pair $(x_n, y_n)$, and using the bound on the Sobolev norm in $\Omega_d$, we obtain
$$ \frac{1}{2} \|u_n\|_\infty = |u_n(x_n) - u_n(y_n)| \le C h(\|u_n\|_\infty)^{1/q} R_n^{2s - N/q}. $$
Rearranging this inequality to isolate the volume factor $R_n^N$, we find
$$ R_n^N \ge C \left( \frac{\|u_n\|_\infty}{h(\|u_n\|_\infty)^{1/q}} \right)^{\frac{Nq}{2sq - N}}. $$
By using the definition of $h$, we have
\begin{equation}\label{Rn:final}
	R_n^N \ge C \|u_n\|_\infty^{-\frac{2N}{N-2s}} \left( L(\|u_n\|_\infty)^{-1} \big( \|u_n\|_\infty |L'(\|u_n\|_\infty)| \big)^{1/q} \right)^{\frac{Nq}{2sq - N}}.
\end{equation}

On the other hand, since the function $t \mapsto t^{\frac{2N}{N-2s}+1}|L'(t)|$ belongs to the class $RV_{\frac{2N}{N-2s}}$ with a positive index of variation, we can apply again the Uniform Convergence Theorem for regularly varying functions (see Theorem~1.5.2. in \cite{Bingham} and Corollary~3.5. in \cite{Cuesta_Pardo_MedJM}) to get
$$ t^{\frac{2N}{N-2s}+1}\,|L'(t)| \geq C \, T^{\frac{2N}{N-2s}+1} \, |L'(T)|, \quad \forall t\in [T/2,T] \ \text{ and } \ T>t_1. $$
Therefore, since $u_n(x) \in [\|u_n\|_\infty/2,\|u_n\|_\infty]$ for each $x\in B_n$, it follows
$$ u_n^{\frac{2N}{N-2s}+1}(x)\,\big|L'(u_n(x))\big| \geq C \, \|u_n\|_\infty^{\frac{2N}{N-2s}+1} |L'(\|u_n\|_\infty)|, \quad \forall x \in B_n. $$
This, together with the radius estimate \eqref{Rn:final}, gives
\begin{align*}
	&\int_\Omega u_n^{\frac{2N}{N-2s}+1}|L'(u_n)| \ge \int_{B_n} u_n^{\frac{2N}{N-2s}+1}|L'(u_n)| \\
	&\hspace{5mm} \ge C \|u_n\|_\infty^{\frac{2N}{N-2s}+1} |L'(\|u_n\|_\infty)| \, R_n^N \\
	&\hspace{5mm}\ge C \|u_n\|_\infty |L'(\|u_n\|_\infty)| \left( L(\|u_n\|_\infty)^{-1} \big( \|u_n\|_\infty |L'(\|u_n\|_\infty)| \big)^{1/q} \right)^{\frac{Nq}{2sq - N}}.
\end{align*}

Finally, by applying Lemma \ref{lem:3.1} to have a uniform lower bound and rearranging the terms to simplify the powers, we conclude that there exists a constant $C > 0$ independent of $n$ such that
$$ \frac{\|u_n\|_\infty |L'(\|u_n\|_\infty)|}{L(\|u_n\|_\infty)^{\frac{N}{2s}}} \le C. $$
However, the structural hypothesis \eqref{hypothesis1:intro} states that this ratio must tend to $+\infty$ as $\|u_n\|_\infty \to +\infty$. This yields a contradiction, which comes from assuming that the sequence $\{u_n\}_n$ is unbounded in $L^\infty(\Omega)$.
\end{proof}

Next, let us prove the result in the strictly subcritical regime.

\begin{proof}[Proof of Theorem \ref{th:aprioriSubcritical:intro}]
First, let us note that the conclusions of Step 1 and Step 2 in the proof of the slightly subcritical case (Lemma~\ref{lem:3.1}) remain valid. That is, there exists a uniform constant $C>0$ such that
\begin{equation} \label{Step2estimate}
    \left| \frac{2N}{N-2s} \int_\Omega F(u) \, dx - \int_\Omega u f(u) \, dx \right| \le C.
\end{equation}

Next, from identity~\eqref{eq:usfF} in the Uniform Convergence Theorem for regularly varying functions (Theorem~\ref{usfF}) applied to $f\in RV_q$, we obtain the pointwise estimate
\begin{equation} \label{SubcriticalFunctEstimate}
    \frac{2N}{N-2s} F(t) - t f(t) \geq \frac{2^*_s-q}{2(q+1)}\,tf(t)-C_1,
\end{equation}
for some constant $C_1>0$ depending only on the nonlinearity $f$.

Now, combining \eqref{Step2estimate} and \eqref{SubcriticalFunctEstimate} yields
$$ [u]_{H^s(\Omega)}^2 \leq [u]_{H^s(\mathbb{R}^N)}^2 = \frac{1}{C_{N,s}} \int_\Omega u f(u) \, dx \leq C,$$
for a positive constant $C$ independent of the solution $u$.

Finally, thanks to the strictly subcritical growth condition, we can apply a bootstrap argument. Relying on the improvement of differentiability provided by Theorem~\ref{th:local:reg}, together with the fractional Sobolev and Morrey embeddings (Theorem~\ref{th:fr:Sobolev:emb}), we obtain the desired uniform $L^\infty(\Om)$ estimate in the interior of the domain. To conclude the proof, it is enough to recall the uniform $L^\infty(\Om)$ estimate close to the boundary given by Theorem~\ref{th:bdd}.
\end{proof}

\bibliographystyle{abbrv}
\bibliography{ref}
\end{document}